\documentclass[11pt,a4paper]{article}
\usepackage{latexsym}
\usepackage[all]{xy}
\usepackage{amsmath,amssymb,amsfonts}
\usepackage{amssymb, graphicx}
\usepackage[latin1]{inputenc}
\usepackage{amsfonts}
\usepackage{latexsym}
\usepackage{amscd}
\usepackage{enumerate}
\RequirePackage{color}
\input xy
\xyoption{all}

\usepackage{multicol}
\usepackage{dsfont}
\usepackage{hyperref}

\usepackage{showlabels}

\hypersetup{ colorlinks,
linkcolor=blue,
filecolor=green,
urlcolor=blue,
citecolor=blue }
\usepackage[top=2cm,bottom=2cm,left=1.5cm,right=1.5cm]{geometry}
\renewcommand{\baselinestretch}{1}
\newtheorem{prethm}{{\bf Theorem}}

\newenvironment{thm}{\begin{prethm}{\hspace{-0.5
               em}{\bf.}}}{\end{prethm}}

\newtheorem{prepro}[prethm]{Proposition}

\newtheorem{prelem}[prethm]{Lemma}
\newenvironment{lem}{\begin{prelem}{\hspace{-0.5
               em}{\bf.}}}{\end{prelem}}

\newtheorem{precor}[prethm]{Corollary}

\newtheorem{preque}[prethm]{Question}
\newenvironment{que}{\begin{preque}{\hspace{-0.5
               em}{\bf.}}}{\end{preque}}
\newtheorem{precon}[prethm]{Conjecture}
\newenvironment{con}{\begin{precon}{\hspace{-0.5
               em}{\bf.}}}{\end{precon}}

\newtheorem{preremark}{{\bf Remark}}

\newtheorem{preexample}{{\bf Example}}

\newtheorem{preproblem}{{\bf problem}}

\newtheorem{preproof}{{\bf Proof.}}

\newenvironment{proof}[1]{\begin{preproof}{\rm
               #1}\hfill{$\Box$}}{\end{preproof}}

\renewcommand{\thefootnote}

\begin{document}
\title{Self-Invariant Maximal Subfields and Their Connexion with Some Conjectures in Division Rings}
\author{M. Aaghabali$^{\,\rm a}$, M.H. Bien$^{\,\rm b}$\\
{\footnotesize{\em $^{\rm a}$ School of Mathematics, Statistics and Computer Science, University of Tehran,
Tehran, Iran}}\\
{\footnotesize {\em $^{\rm b}$Faculty of Mathematics and Computer Science,
University of Science, VNU-HCM,
227 Nguyen Van Cu Str.,}}\\
{\footnotesize{\em  Dist. 5, HCM-City, Vietnam}}}

\footnotetext{E-mail Addresses:  $^{\rm a}${\tt mehdi.aaghabali@ut.ac.ir, maghabali@gmail.com}, $^{\rm b}$~{\tt mhbien@hcmus.edu.vn}\\ The authors would like to thank Mehran Motiee for his fruitful discussions and helpful comments. The first author acknowledges the INEF grant 380014. The first author is indebted to the Research Council of University of Tehran for support. The second author was funded by Vietnam National University HoChiMinh City (VNU-HCM) under grant no. C2018-18-03.}
\date{}
\maketitle
%------------------------------------------------------------------------------------------------------------------------------
\begin{quote}
{\small \hfill{\rule{13.3cm}{.1mm}\hskip2cm} \textbf{Abstract}
Let $D$ be a division algebra with center $F$. A maximal subfield of $D$ is defined to be a field $K\subseteq D$ such that $C_D(K)=K,$ that is, $K$ is its own centralizer in $D.$ A maximal subfield $K$ is said to be self-invariant if it normalises by itself, i.e. $N_{D^*}(K)\cup\{0\}=K.$ This kind of subfields is important because they have strong connexion with most famous Albert's Conjecture (every division ring of prime index is cyclic). In fact, we pose a question that asserts whether every division ring whose all maximal subfields are self-invariant has to be commutative. The positive answer to this question, in finite dimensional case, implies the Albert's Conjecture (see \S2). Although we show the Mal'cev-Neumann division ring demonstrates negative answer in the case of infinite dimensional division rings, but it is still most likely the question receives positive answer if we restrict ourselves to the finite dimensional division rings.  We also have had the opportunity to use the Mal'cev-Neumann structure to answer Conjecture~\ref{c1} below in negative (see \S3). Finally, among other things, we rely on this kind of subfields to present a criteria for a division ring to have finite dimensional subdivision ring (see \S4).
\vspace{1mm} {\renewcommand{\baselinestretch}{1}
\parskip = 0 mm

}

}

\vspace{-3mm}\hfill{\rule{13.3cm}{.1mm}\hskip2cm}
\end{quote}
%-------------------------------------------------------------------------------------------------------------------------------
%================================================================================================================================
\section{introduction}
Let $D$ be a division ring with center $F.$  An element $a\in D$ is called \textit{algebraic}   over $F$, if there exists a non-zero polynomial $a_0+a_1x+\dots+a_nx^n$ over $F$ such that $a_0+a_1a+\dots+a_na^n=0.$ If $a\in D$, then $F(a)$ denotes the subfield of $D$  generated by $F$ and $\{a\}$. If $A$ is a subset of $D,$ then $\Delta(A)$ denotes the division ring generated by $\Delta\cup A.$ If $A$ is a subset of $D$ we use $A^*$ to denote $A\setminus\{0\}.$ Also, we denote by ${\rm Char}(D), {\rm dim}_FD$ and ${\rm Aut}(\Delta)$ the characteristic of $D,$ the dimension of $D$ over $F$ and automorphisms of subfield $\Delta,$ respectively.
Let $A,B$ be two subsets of $D.$ We say that $A$ is $B$-\textit{invariant} if $bAb^{-1}\subseteq A,$ for all $b\in B^*.$ Also, by the \textit{normalizer} of $A$ in $B$ we mean the set $N_{B^*}(A)=\{b\in B^*;bAb^{-1}\subseteq A\}.$ For a subset $A\subseteq D$ the \textit{centralizer} of $A$ in $D$ is defined to be the set $C_D(A)=\{d\in D;da=ad {\rm~for~all~}a\in A\}.$ Clearly, $C_D(A)$ is a division subring of $D$ and $N_{D^*}(A)$ is a multiplicative subgroup of $D^*.$ A \textit{maximal subfield} of $D$ is defined to be a field $K\subseteq D$ containing $F$ such that $C_D(K)=K,$ that is, $K$ is its own centralizer in $D.$ Let $G$ be group and $H$ a subgroup of $G$. If there exists a sequence of subgroups $$G=H_0\trianglerighteq H_1\trianglerighteq H_2\trianglerighteq \cdots \trianglerighteq H_n=H,$$ then we say that $H$ is \textit{subnormal} in $G$. Additionally, if $n$ is the smallest number in such all sequences, then $H$ is called \textit{$n$-subnormal} in $G$. Hence, the $0$-subnormal subgroup of $G$ is $G$, every $1$-subnormal subgroups of $G$ is normal in $G$ and every $n$-subnormal subgroup of $G$ is non-normal in $G$ for every $n\ge 2$.

There are wild studies on certain substructures of a division ring including a lot of conjectures on the structural properties arising from these substructures (see \cite{aa1,aa2} and references therein). For a given division ring considering an special property $P,$ say commutativity, algebraicity or other finiteness conditions, it is common to understand whether one can specify a set or a substructure $S,$ such that the property $P$ for $S$ implies the property $P$ for the whole ring. Often times one can even go beyond and ask how the problem will proceed if the condition $P$ over the substructure causes the condition $Q$ over the entire structure.  There are several candidates for set or substructure $S,$ including the generator sets, constructions related to commutators,  normal or subnormal subgroups of the multiplicative group of a division ring, and (skew) symmetric elements in rings with involution, just to be mentioned a few.  The first departure towards these kinds of research was done by Wedderburn in 1905 in which he connected the algebraic properties of $D $ to that of $D^*,$ the multiplicative group of division ring that states all finite division rings are commutative~\cite[p. 203]{lam}. Afterwards, many other authors became interested in this stuff and in particular Kaplansky showed that every division ring whose all elements are periodic modulo the center is commutative. These results were stimulated by Jacobson's famous result, that asserts every division ring algebraic over a finite field is commutative, and were extensively developed by Herstein based on prominent role of ``commutator" structures in division rings. Herstein conjectured in Kaplansky's result, it suffices to examine the periodicity condition for multiplicative commutators rather than playing with all elements. He himself proved the conjecture in certain cases, namely, when the division ring is centrally finite, when the multiplicative commutators have finite order in multiplicative subgroup of division ring, and in the case of division rings with uncountable centers~\cite[p. 210]{lam}. There are other partial answers to this conjecture, however, the general case is still remaind open!

Some of efforts to answer the Herstein's Conjecture led people to try structures generated by the multiplicative commutators, say the group of multiplicative commutators, and in more general setting to examine the conjecture by replacing normal and subnormal subgroups instead of commutator subgroups. During the attempts to prove the conjecture, algebraists reproduced similar statements in different situations applying normal and subnormal subgroups. For example, in 1978, Herstein conjectured that every normal subgroup of the multiplicative subgroup of a division ring that is radical over the center, has to be central~\cite{her1}. He showed that the conjecture is true if division ring is centrally finite, but in general it is still open, however, it has been shown that if $N$ is a subnormal subgroup of $D^*$ that is periodic modulo the center of $D,$ then $N$ is central. The last result is a special type of another absorbing old problem determining how much subnormal subgroups of $D^*$ reflect the multiplicative structure of $D^*.$ In other words, how ``big" subnormal subgroups are in $D^*.$ The most important result concerning the structure of subnormal subgroups was obtained by Stuth in 1964 asserting that (i) If $G$ is a noncentral subnormal subgroup of $D^*$ and $x^G$ is the conjugacy class of the noncentral element $x\in D^*$ in $G,$ then the division subring generated by $x^G$ is $D,$ (ii) Every soluble subnormal subgroup of $D^*$ is central. In recent years,  for a division ring $D,$ there has been renewed interest in the study of normal subgroups of $D^*$ \cite{mai,Pa_Go_17,Pa_Go_Pa_15,Pa_HaDeBi_12}. Herstein and Scott conjectured that every subnormal subgroup of $D^*$ is normal in $D^*$ (see \cite[Page 80]{her}) and it was shown in \cite{Pa_Gr_78} that this conjecture of Herstein and Scott holds for the real quaternion division ring, but not for a finite dimensional division ring $D$ over a $p$-local field. Hence, it is natural to extend all results concerning normal subgroups to subnormal subgroups of $D^*,$ and one can pose ``weak version'' of Herstein-Scott's conjecture as follows:

\begin{con} {\rm \cite[Conjecture A]{Pa_Gr_81}}\label{c1}
	If $D$ is a division ring, then every non-central subnormal subgroup of the multiplicative group $D^*$ of $D$ contains a non-central normal subgroup of $D^*$. 	
\end{con}

The motivation of Conjecture~\ref{c1} is from results on subnormal subgroups of general skew linear groups ${\rm GL}_m(D)$. In detail, it is well known that (or see \cite{Pa_MaAk_98}) if $D$ is a division ring containing at least 5 elements and $m>1$, then every non-central subnormal subgroup of the general skew linear group ${\rm GL}_m(D)$ is normal. However, in the case when $m=1$, that is, ${\rm GL}_1(D)=D^*$, there are division rings $D$ whose multiplicative groups $D^*$ contain non-normal subnormal subgroups \cite{Pa_Gr_78}.  If Conjecture~\ref{c1} holds, then one would extend trivially several results on normal subgroups for subnormal subgroups in division rings (e. g., results in \cite{mai,Pa_Go_17, Pa_Go_Pa_15, Pa_HaDeBi_12}). The Conjecture~\ref{c1} is affirmative in case $D$ is finite dimensional over a $p$-local field with $p\ne 2$ \cite{Pa_Gr_81}.

 In this note, we show that if $D=\Delta((G,\sigma))$ is the Mal'cev-Neumann division ring of a non-cyclic free group $G$ over a division ring $\Delta$ with respect to a group morphism $\sigma : G\to {\rm Aut}(\Delta)$ and if $n$ is a positive integer, then there exists an $n$-subnormal subgroup $N$ of $D^*$ such that for every $0\le \ell<n$, $N$ contains no non-central $\ell$-subnormal subgroup of $D^*$. It means that the answer to Conjecture~\ref{c1} is negative in general (see Theorem~\ref{t5}).

The other substructures in a division ring which one can exploit a lot of information about the entire ring looking after them are  {\it maximal subfields}. For example, the following most well-known result asserts one can obtain the finite dimensionality of a division ring from finite dimensionality of its maximal subfields.

\begin{thm}{\rm \cite[P. 242]{lam}}\label{dim}
If $D$ is a division algebra of degree $n$ over a field $F$ and $K$ is a subfield of $D$ containing $F,$ then ${\rm dim}_FK\le n$. The equality holds if and only if $K$ is a maximal subfield of $D$. Conversely, if $K$ is a maximal subfield of a division ring $D$ such that ${\rm dim}_FK=n<\infty,$ then ${\rm dim}_FD=n^2.$
\end{thm}

There are still other evidences to show how one can rely on maximal subfields to deduce further information about the whole division ring \cite{aa3}. In particular, according to the Skolem-Noether Theorem \cite[P. 93]{farb} we know that if $A$ is a centrally finite simple $F$-algebra, then every $F$-automorphism of $A$ is inner. Hence, if $\Delta$ is a division subring of a centrally finite division ring, then every automorphism $\varphi\in{\rm Aut}(\Delta)$ is inner. Thus every element of the form $x^{-1}\varphi(x),$ known as {\it autocommutator}, is nothing else a multiplicative commutator in $D^*.$ By \cite[P. 210]{lam}, it is known that
 \begin{center}{{\it every division ring whose all multiplicative commutators are central is a field}~~~~~~~~~~~~($\dag$)}
  \end{center} Now, assume that $\Delta\subseteq D$ is an extension of division rings. Note that here we have no dimension restrictions on $D.$ We claim that if there exists a non-trivial $\sigma\in {\rm Aut}(D)$ such that $x^{-1}\sigma(x)\in\Delta$ for any $x\in D^*$, then $D=\Delta.$ Indeed, let $H=\{a\in D\mid \sigma(a)=a\}$ be the fixed division subring of $D$ with respect to $\sigma$. One has $H\neq D,$ since $\sigma\neq {\rm id}.$ By a well-known result in linear algebra $D$ could not be equal to a finite union of its proper vector subspaces. So $D=H\cup\Delta,$ implies $D=\Delta.$ We claim that the relation $D=\Delta$ holds in general. In fact, by the contrary suppose that $D\setminus (H\cup\Delta)\neq \emptyset$ and let $x\in D\setminus (H\cup \Delta).$ It is clear that $x\neq 0,-1$. Now, $$(x+1)[x^{-1}\sigma(x)-(x+1)^{-1}\sigma(x+1)]=x^{-1}\sigma(x)-1\in \Delta.$$ Since $x^{-1}\sigma(x)\neq 1$, the element $x^{-1}\sigma(x)-(x+1)^{-1}\sigma(x+1)$ is a non-zero element of $\Delta.$ Therefore, $x+1\in \Delta$ and as a result, $x\in \Delta$, which contradicts the hypothesis. So, if we consider a division ring $D$ with center $F$ and a maximal subfield $L$ of $D,$ then putting  $D=L$ and $\Delta=F$ in above argument, one has $L=F$ which implies the following generalization of $\dag.$
\begin{thm}
Let $D$ be a division ring with center $F$ and $L$ a maximal subfield of $D$. If there exists non-trivial $\sigma\in {\rm Aut}(L)$ such that $x^{-1}\sigma(x)\in F$ for any $x\in L^*$, then $D$ is commutative.
\end{thm}
%---------------------------------------------------------------------------------------------------------------------

 Now, a question is naturally proposed: Which other properties of a division ring may arise from its maximal subfields$?$ At the first step to generalize Theorem~\ref{dim}, one might be interested in thinking of following question.  Whether can one conclude the algebraicity of a division ring over its center from the algebraicity of its maximal subfields over the center$?$ The second named author and D.H. Dung in \cite{mai}, could provide a non-algebraic division ring which contains a maximal subfield algebraic over the center and answer the question in negative. In fact, using the Mal'cev-Neumann's construction of Laurent series rings, they present an example of a division ring $D$ whose all elements are Kurosh, that is for each element $x\in D$ there exists a
centrally finite division subring of $D$ containing $x$ as a non-central element, but $D$ is not algebraic.  Nevertheless,  the question still could achieve positive answer if we assume all maximal subfields are algebraic over the center. In fact,

\begin{con}
Every division ring whose all maximal subfields are algebraic over the center is algebraic.
\end{con}

Note that for a given division ring, the property of being algebraic over every maximal subfield seems rather similar to the property of being algebraic over the center. Indeed, it is not known examples of division rings of either phenomenon except for division rings which are locally PI. Thus it could not be appealing if we consider the situation under which a division ring is algebraic over every maximal subfield and ask whether division ring is algebraic over the center.

In the sequel we will study a special kind of subfields in a division ring so-called self-invariant subfields. A subfield $K$ of $D$ is called \textit{self-invariant} if $N_{D^*}(K^*)\cup\{0\}=K.$  In first instance one might be interested in learning whether such kind of structures really does exist. The purpose of the preceding argument is to somehow prove the existence of such structures in division rings. Indeed, let $D$ be a centrally finite division ring over its center $F$ and ${\rm Char}(F)=p>0.$ Assume that $K$ is a purely inseparable maximal subfield in $D.$ We know that the minimal polynomial of every element $a\in K$ is of the form $f(x)=(x-a)^{p^n},$ for some integer $n.$ Now, if $\sigma\in {\rm Aut}(K),$ then $\sigma(a)$ is also a root of $f(x).$ Hence, for every $a\in K,$ we have that $\sigma(a)=a$ and consequently ${\rm Gal}(K/F)=\{id\}.$ We claim that $K$ is a self-invariant subfield in $D,$ in fact otherwise, applying the Skolem-Noether Thorem one can establish a non-trivial automorphism to $K$ under which the central elements are invariant.   J.M. Bois and G. Vernik in \cite[Theorem 2.1.5]{bois} proved that if $F$ is a field with ${\rm Char}(F)=p>2$ and $L$ is a finite dimensional non-abelian solvable Lie algebra over $F,$ then $K(L),$ the division ring of fractions of the enveloping algebra $U(L),$ contains some purely inseparable maximal subfield. There are more results regarding self-invariant subfields of division rings in \cite{aa}.

 The main goal is here to provide answer to following question:
\begin{que}\label{que}
Let $D$ be a division ring with center $F$ whose all maximal subfields are self-invariant. Can one conclude that $D$ is commutative$?$
\end{que}

In this note we show that all maximal subfields of the Mal'cev-Neumann division ring of a non-cyclic free group over a field with respect to a group morphism, which is an infinite dimensional division ring, are self-invariant, and this denies the above question in the case of infinite dimensional division rings.
\bigskip
%=========================================================================================================================

\section{Self-Invariant Subfields}

The Skolem-Noether Theorem says that in the finite dimensional division rings, a subfield is not self-invariant if and only if contains some non-trivial automorphisms. Therefore, in general if $D$ is a finite dimensional division ring with center $F$ and $K$ is a maximal subfield of $D$ with ${\rm Gal}(K/F)\neq \{id\},$ then $K$ is not self-invariant. In other words one can conclude that if $D$ is a division ring with center $F$ which contains some self-invariant subfields $K$ such that ${\rm Gal}(K/F)\neq \{id\},$ then ${\rm dim}_F(D)=\infty.$ In other extreme, when $D$ is finite dimensional over its center looking at natural surjective homomorphism of groups, $\phi:N_{D^*}(K^*)\longrightarrow {\rm Gal}(K/F),$ one can find that ${\rm ker}(\phi)=K^*$ and consequently obtain that $N_{D^*}(K^*)/K^*\simeq {\rm Gal}(K/F).$ Hence, if the Question~\ref{que} achieves answer in affirmative, then every non-commutative division ring finite dimensional over its center admits a maximal subfield, say $K,$ with non-trivial ${\rm  Gal}(K/F).$ This gives also a positive answer to the certain case of a problem due to Mahdavi that asserts whether every division ring algebraic over its center $F$ contains a subfield $K$ with non-trivial ${\rm Gal}(K/F)$~\cite[Problem 10, P. 82]{Pa_Ma_00}. In the special case when the division ring is of prime index, one can easily see that division ring contains a maximal subfield with non-trivial ${\rm Gal}(K/F)$ if and only if division ring is cyclic, that is, division ring contains a maximal subfield which is cyclic Galois over the center $F.$ Therefore, if the Question~\ref{que} finds positive answer one can show that every division ring of prime index is cyclic. This is Albert's well-known conjecture that is still unsolved!

By the celebrated Brauer-Cartan-Hua Theorem, one can easily verify that if $K$ is a maximal subfield of a non-commutative division ring $D,$ then $N_{D^*}(K^*)\subsetneq D.$ It is known that every maximal subfield of a division ring is also maximal with respect to inclusion. Hence, if $K$ is a self-invariant subfield of $D,$ then $K$ is a maximal subfield of $D.$ Since otherwise we can find a subfield $T$ such that $K\subsetneq T.$ Therefore, $T\subseteq N_{D^*}(K^*)=K^*$ is a contradiction. Note that the converse is not true. Namely, let $D$ be the cyclic algebra $(\mathbb{Q}(v)/\mathbb{Q},\sigma,2)$ with $v^3-3v+1=0$ and $\sigma(v)=v^2-1.$ By \cite[P. 238, Exercise 14.16]{lam}, $\mathbb{Q}(v)$ is a cyclic cubic field extension over $\mathbb{Q},$ and $D$ is a $9$-dimensional division algebra over its center $\mathbb{Q}.$ Then $D$ contains a subfield $K$ isomorphic to $\mathbb{Q}(a)$ with $a^3 = 2.$ The field $K$ has ${\rm Gal}(K/\mathbb{Q})=\{id\},$ so it is a self-invariant subfield in $D.$ Note that $K$ is a self-invariant subfield in $D$ with no non-central proper subfield. On the other hand, since $D$ is cyclic, it admits a maximal subfield $K$ with ${\rm Gal}(K/\mathbb{Q})\neq\{id\}.$ Thus, $D$ also contains infinitely many subfields which are not self-invariant.

If $K$ assumed to be  a self-invariant maximal subfield in $D,$ then there exists an element $x\in D\setminus K.$ Hence, $xKx^{-1}$ is another self-invariant maximal subfield in $D$ different from $K.$ Now, consider $t\in K\setminus N_{D^*}(xKx^{-1}).$ One can observe that $txKx^{-1}t^{-1}$ is a self-invariant maximal subfield different from $xKx^{-1}.$ Also, $txKx^{-1}t^{-1}$ is different from $K,$ because otherwise $tx\in K$ and $x\in K,$  contradiction! Note that $K\cup xKx^{-1}\nsubseteq txKx^{-1}t^{-1},$ since otherwise $xKx^{-1}\subseteq C_D(K)=K$ yields a contradiction. Thus one can consider $t'\in K\cup xKx^{-1}\setminus txKx^{-1}t^{-1},$ and applying aforementioned argument find a self-invariant maximal subfield $t'txKx^{-1}t^{-1}t'^{-1}$ distinct from $K,xKx^{-1},$ and $txKx^{-1}t^{-1}.$ In shorts, if a division ring contains a self-invariant subfield, then it admits infinitely many self-invariant maximal subfields. Hence, we can state following result.

\begin{lem}\label{mehdi} Let $D$ be a division ring and $K$ a maximal subfield of $D$. Then, $K$ is self-invariant if and only if so is $aKa^{-1}$ for every $a\in D^*$.
\end{lem}

In this section we provide a division ring infinite dimensional over its center that demonstrates negative answer to Question~\ref{que}. More precisely, we investigate maximal subfields of Mal'cev-Neumann division ring and show all of them are self-invariant, however, the division ring is not commutative. Let start with some preliminary facts in order to introduce the structure of Mal'cev-Neumann division ring.

For a subset $S$  of a group $G$, we define a sequence of subgroups $\langle S\rangle_n,n\in {\mathbb N},$ of $G$ containing $S$ as follows: put $\langle S\rangle_0=G$. For $n>0$, the subgroup $\langle S\rangle_n$ is the normal closure of $S$ in $\langle S\rangle_{n-1}$. Then, we have a sequence of normal subgroups $$G=\langle S\rangle_0\trianglerighteq \langle S\rangle_1\trianglerighteq \langle S\rangle_2\trianglerighteq \cdots.$$ In general, for every natural number $n$, the subgroup $\langle S\rangle_n$ is subnormal in $G$ and unnecessary $n$-subnormal in $G$.
A recent result, in 2017, showed that if $G$ is a non-cyclic free group and $g\in G\backslash\{1\}$, then $\langle g\rangle_n$ is an $n$-subnormal subgroup of $G$ \cite[Corollary 1.3]{Pa_Ol_17}. For convenience, we will use the following form of \cite[Corollary 1.3]{Pa_Ol_17}.

\begin{lem}\label{l1} {\rm\cite[Corollary 1.3]{Pa_Ol_17}} Let $G$ be a non-cyclic free group and $g\in G\backslash \{1\}$. Assume $N$ is a non-trivial subgroup of $\langle g\rangle_{n}$. If $N$ is $\ell$-subnormal in $G$, then $\ell\ge n$.
\end{lem}

We also borrow the following result.
\begin{lem}\label{normalizer}{\rm \cite[Exersice 7, P. 42]{Bo_MaKaSo_76}} Let $G$ be a free group and $h\in G\backslash \{1\}$. The normalizer $N_G(\langle h\rangle)$ of the subgroup $\langle h\rangle$ in $G$ is abelian.
\end{lem}

Recall that a \textit{total ordered group} $G$ is a (non-abelian or abelian) group with a total order $\preceq$ such that for every $a,b,c\in G$, if $a\preceq b$, then $ac\preceq bc$ and $ca\preceq cb$. It is well-known that the class of total ordered groups includes free groups with dictionary order. For a total ordered group $G$, a subset $S$ of $G$ is called \textit{well-ordered} (briefly, \textit{WO}) if every non-empty subset of $S$ has a least element. For a  WO non-empty subset $S$  of $G$, we denoted by $\min(S)$ the least element in $S$.

Now, let $\Delta$ be a division ring, let $G$ be a total ordered group and let $\sigma : G\to {\rm Aut}(\Delta), g\mapsto \sigma_g,$ be a group morphism. We consider formal sums of the form $\alpha=\sum\limits_{g\in G}a_gg,$ where $a_g\in \Delta$. For such $\alpha$, we define ${\rm supp}(\alpha)=\{g\in G\mid a_g\ne 0\}$ and call it the {\it support} of $\alpha$. Put $$\Delta((G,\sigma))=\left\{\alpha=\sum\limits_{g\in G}a_gg\mid {\rm supp}(\alpha) \text{ is WO } \right\}.$$ For every $\alpha=\sum\limits_{g\in G}a_gg, \beta=\sum\limits_{g\in G}b_gg\in \Delta((G,\sigma))$, we define $$\alpha+\beta=\sum\limits_{g\in G}(a_g+b_g)g$$ and $$\alpha\beta=\sum\limits_{t\in G}\left( \sum\limits_{gh=t} a_g\sigma_g(b_h)\right)t.$$ The above operators are well-defined \cite{lam} and moreover,

\begin{lem}{\rm \cite[Theorem  14.21]{lam}}\label{l2}
	$\Delta((G,\sigma))$ is a division ring.
\end{lem}
The division ring $\Delta((G,\sigma))$ is called the \textit{Mal'cev-Neumann division ring} of $G$ over $\Delta$ with respect to $\sigma$.
The following lemma is useful for not only this section but also next one.

\begin{lem}\label{l3} Let $D=\Delta((G,\sigma))$ be the Mal'cev-Neumann division ring of a total order group $G$ over a division ring $\Delta$ with respect to a group morphism $\sigma : G\to {\rm Aut}(\Delta)$. Put $v : D^*\to G, \alpha\mapsto \min({\rm supp}(\alpha)).$ Then,
	\begin{enumerate}
		\item The function $v$ is a surjective group morphism.
		\item For every $\alpha, \beta\in D$, if $v(\alpha)\ne v(\beta)$, then $v(\alpha+\beta)=\min\{v(\alpha), v(\beta)\}$.
	\end{enumerate}
	
\end{lem}
\begin{proof} {1. If $\alpha=\sum\limits_{g\in G}a_gg \in D^*$, then ${\rm supp}(\alpha)$ is a non-empty subset of $G$, which implies that $v$ is well-defined.
		Now, let $\alpha=\sum\limits_{g\in G}a_gg, \beta=\sum\limits_{g\in G}b_gg\in D^*$ with $g_{\alpha}=\min({\rm supp}(\alpha))$ and  $g_{\beta}=\min({\rm supp}(\beta))$. Since $$\alpha\beta=\sum\limits_{t\in G}\left( \sum\limits_{gh=t} a_g\sigma_g(b_h)\right)t,$$ the support of $\alpha\beta$ is $${\rm supp}(\alpha\beta)=\{t\in G\mid \sum\limits_{gh=t} a_g\sigma_g(a_h)\ne 0\}.$$ Observe that for every $g,h\in G$ with $a_g\ne 0$ and $b_h\ne 0$, one has $g_\alpha\le g$ and $g_\beta\le h$, so $g_\alpha g_\beta\le gh$. Moreover, $a_{g_\alpha}\sigma_g(b_{g_\beta})\ne 0$ which implies that $\min({\rm supp}(\alpha\beta))=g_\alpha g_\beta$. Hence, $v(\alpha\beta)=v(\alpha)v(\beta)$. The group morphism $v$ is surjective trivially since for every $g\in G$, one has $v(g)=g$. Thus, the proof of (1) is complete.
		
		2. The proof of this assertion is easy and we left it to the reader. }	
\end{proof}

Here, we focus on a special case where $\Delta$ is a field and $G$ is a free group with Magnus order (as in the proof of \cite[Theorem 6.31]{lam}). First, we need to know the center of $\Delta((G))=\Delta(G,Id)$.
\begin{lem}{\rm \cite[Corollary 14.26]{lam}}\label{center} If $G$ is a free group of rank $\ge 2$ with dictionary order  and $\Delta$ is a field, then $Z(\Delta((G)))=\Delta$.
\end{lem}
The following result is very important for this note.
\begin{lem} \label{cohn} Let $G$ be a free group with dictionary order and $\Delta$ be a field. For every $\alpha\in \Delta((G))$ with $h=v(\alpha)>1$, there exists non-zero element $\beta\in \Delta((G))$ such that  $$\beta\alpha\beta^{-1}=a_mh^m+a_{m+1}h^{m+1}+\cdots,$$ where $m\in \mathbb{Z}$, $a_i\in \Delta$ for every $i\ge m$.
	\end{lem}
\begin{proof}
{According to \cite[Theorem 1.5.11]{cohn}, there exists a non-zero element $\beta\in \Delta((G))$ such that ${\rm supp}(\beta\alpha\beta^{-1})\subseteq C_G(h)$. Observe that $G$ is free, so the centralizer $C_G(h)=\langle h\rangle$. Hence, since $h>1$,  $$\beta\alpha\beta^{-1}=\sum\limits_{i=m}^\infty a_ih=a_mh^m+a_{m+1}h^{m+1}+\cdots,$$ where $m\in \mathbb{Z}$, $a_i\in K$ for every $i\ge m$.}
\end{proof}

In this section, $\Delta$ is assumed to be a field, $G$ a free group of rank $\ge 2$ with dictionary order and $D=\Delta((G))$ the Mal'cev-Neumann division ring of $G$ over $\Delta$.

\begin{lem}{\rm \cite[Exercise 1.5.8]{cohn}}\label{centralizer} Let $h\in G$ be such that $h>1$ and $\alpha=\sum\limits_{i=m}^\infty a_ih^i$ be an element in $D\backslash \Delta$. Then $$C_D(\alpha)=\{\sum\limits_{i=n}^\infty b_ih^i\mid n\in \mathbb{Z}, b_i\in \Delta\text{ for every } i\ge n \}.$$
	\end{lem}

\begin{thm}
	Every maximal subfield of $D$ is self-invariant.
\end{thm}
\begin{proof}
{Assume that $L$ is a maximal subfield of $D$. Since $D$ is non-commutative and $\Delta=Z(D)$ (Lemma~\ref{center}), $L\backslash \Delta\ne\emptyset$. Let $\alpha=\sum\limits_{g\in G}a_gg\in L\backslash \Delta$. Clearly, $\alpha^{-1}$ and $\alpha-a_1\in L\backslash \Delta$. Moreover, if $v(\alpha)\le 1$, then $v(\alpha^{-1})\ge 1$ and if $v(\alpha)=1$, then $v(\alpha-a_1)>1$.  Hence, without loss of generality, we assume that $h=v(\alpha)>1$. According to Lemma~\ref{cohn}, there exists a non-zero element $\beta\in D$ such that $$\alpha_1=\beta\alpha\beta^{-1}=\sum\limits_{i=m}^\infty a_ih^i=a_mh^m+a_{m+1}h^{m+1}+\cdots,$$ where $m\in \mathbb{Z},~a_i\in \Delta$ for every $i\ge m$. Now, we consider the maximal subfield $L_1=\beta L\beta^{-1}$. We claim that $$L_1= \{\sum\limits_{i=m}^\infty a_ih^i \mid m\in \mathbb{Z}, a_i\in \Delta \text{ for every } i\ge m\}.$$ 	One has $C_D(\alpha_1)\supseteq L_1$ since $\alpha_1\in L_1$. By Lemma~\ref{centralizer}, $$C_D(\alpha_1)=\{\sum\limits_{i=n}b_ih^i\mid n\in \mathbb{Z}, b_i\in \Delta\text{ for every } i\ge n \},$$ which is a subfield of $D$. Since the maximality of $L_1$, one has $$L_1=C_D(\alpha_1)= \{\sum\limits_{i=m}^\infty a_ih^i \mid m\in \mathbb{Z}, a_i\in \Delta \text{ for every } i\ge m\}.$$
	
	Now, we claim that $L_1$ is self-invariant. Let $\gamma=\sum\limits_{g\in G}c_gg\in D^*$ such that $\gamma L_1\gamma^{-1}\subseteq L_1$. It suffices to show ${\rm supp}(\gamma)\subseteq \langle h\rangle$. Put $A={\rm supp}(\gamma)$ and assume that $A\backslash \langle h\rangle\ne \emptyset$. Now,
	let $\gamma=\delta+\epsilon,$ where $\delta=\sum\limits_{g\in \langle h\rangle}c_gg$ and $\epsilon=\sum\limits_{g\in A\backslash \langle h\rangle } c_gg$. Clearly, $\delta\in L_1$ and $g_\epsilon=v(\epsilon)\notin \langle h\rangle$. For every integer $\ell$, if $\lambda = \gamma h^\ell\gamma^{-1}=\sum\limits_{i=m}^\infty a_ih^i\in L_1$ with $a_m\ne 0$, then $\epsilon h^\ell-\lambda\epsilon =\lambda\delta-\delta h^\ell\in L_1$. There are three cases:
	\bigskip
	
	{\it Case 1. $v(\epsilon h^\ell)>v(\lambda\epsilon)$.} Then by Lemma~\ref{l3}, $v(\epsilon h^\ell-\lambda\epsilon)=v(\lambda\epsilon)=v(\lambda)v(\epsilon)=h^mg_\epsilon$. Observe that $h^mg_\epsilon\not \in \langle h\rangle$ so $v(\epsilon h^\ell-\lambda\epsilon)\not\in \langle h\rangle,$ which contradicts the fact that $\epsilon h^\ell-\lambda\epsilon\in L_1$.
	\bigskip
	
	{\it Case 2. $v(\epsilon h^\ell)<v(\lambda\epsilon)$.} Similar to Case 1, this implies that $v(\epsilon h^\ell-\lambda\epsilon)=v(\epsilon h^\ell)=g_\epsilon h^\ell\not\in \langle h\rangle,$ which is also a contradiction.
	\bigskip
	
	{\it Case 3. $v(\epsilon h^\ell)=v(\lambda\epsilon)$.} Then, by Lemma~\ref{l3} (1), $g_\epsilon h^\ell =h^m g_\epsilon$. As $\ell$ ranges over the set of integers, $g_\epsilon\in N_G(\langle h\rangle)$. In the view of Lemma~\ref{normalizer}, $g_\epsilon h=hg_\epsilon$. By Lemma~\ref{centralizer}, $g_\epsilon=h^j$ for some integer $j$. This contradicts the fact that $g_\epsilon\not\in \langle h\rangle$.
	\bigskip
	
	Three cases lead us to a contradiction. Thus, $L_1$ is self-invariant, and consequently applying Lemma~\ref{mehdi} we yield $L$ is also self-invariant. The proof is now complete.}
\end{proof}

%-------------------------------------------------------------------------------------------------------------
\section{Subnormal non-normal subgroups of a division ring}

Although this section seems to be far from the main theme of this paper, that is on self-invariant maximal subfields, but we find it convenient to apply Mal'cev-Neumann structure to give negative answer to Conjecture~\ref{c1}. We start with following lemma.

\begin{thm}\label{t5} Let $G$ be a non-cyclic free group with dictionary order, let $\Delta$ be a division ring and let $\sigma : G\to {\rm Aut}(\Delta)$ be a group morphism. For a positive integer $n$, there exists an $n$-subnormal subgroup $N$ of the multiplicative group $D^*$ of $D=\Delta((G,\sigma))$ such that if $M$ is a subgroup of $N$ which is $\ell$-subnormal in $D^*$, then $\ell\ge n$.
\end{thm}
\begin{proof} {Let $v : D^* \to G, \alpha\mapsto \min({\rm supp}(\alpha)),$ be the group morphism as in Lemma~\ref{l3}. Assume that $x$ is an element of $G\backslash\{1\}$. Consider the sequence  $$G=\langle x\rangle_0\trianglerighteq \langle x\rangle_1\trianglerighteq \langle x\rangle_2\trianglerighteq\cdots\trianglerighteq \langle x\rangle_n$$ of subgroups as in Lemma~\ref{l1}. Put $$N=N_n=v^{-1}(\langle x\rangle_n)=\{\alpha\in D \mid \min({\rm supp}(\alpha))\in \langle x\rangle_n \}.$$ We show that $N$ is an $n$-subnormal subgroup of $D^*$ and for every $\ell<n$, $N$ contains no non-central $\ell$-subnormal subgroup of $D^*$. Indeed, we first show that $N$ is $n$-subnormal in $D^*$. For every $0\le i<n$, put $$N_i=v^{-1}(\langle x\rangle_i)=\{\alpha\in D^*\mid \min({\rm supp}(\alpha))\in \langle x\rangle_i \}.$$ Then, $$N=N_n\trianglelefteq N_{n-1}\trianglelefteq \cdots \trianglelefteq N_1\trianglelefteq N_0=D^*,$$ which implies that $N$ is subnormal in $D^*$. Moreover, assume that $$N=X_m\trianglelefteq X_{m-1}\trianglelefteq \cdots \trianglelefteq X_1\trianglelefteq X_0=D^*$$ is a sequence of normal subgroups of $D^*$. Then,  $$\langle x\rangle _n=v(N)=v(X_m)\trianglelefteq v(X_{m-1})\trianglelefteq \cdots \trianglelefteq v(X_1)\trianglelefteq v(X_0)=v(D^*)=G$$ is a sequence of normal subgroups of $G$. It implies that $\langle x\rangle_n$ is $m'$-subnormal in $G$ with $m\ge m'$. By Lemma~\ref{l1}, $m'\ge n$, so $m\ge m'\ge n$. Hence, $N$ is $n$-subnormal in $D^*$.
	
	Now we show that for every $\ell<n$, the subgroup $N$ contains no non-central $\ell$-subnormal subgroup of $D^*$. Assume that $M$ is a subgroup of $N$ which is non-central $\ell$-subnormal in $D^*$. Then the image $v(M)=\{v(\alpha) \mid \alpha\in M\}$ is a subgroup of $\langle x\rangle_n$. Moreover, if $v(M)$ is trivial, then $M\subseteq \Delta$, so, obviously, $xKx^{-1}\subseteq \Delta$ for every $x\in M$. In the view of the Cartan-Brauer-Hua Theorem for the subnormal version \cite[14.3.8, Page 439]{scott2}, either $\Delta$ is central or $\Delta=D$. The second case is impossible since $G$ is non-trivial, which implies that $\Delta$ is central. In particular, $M$ is central which contradicts to the hypothesis. Hence, $M\backslash \Delta\ne \emptyset$, which implies that $v(M)$ is non-trivial in $\langle x\rangle_n$. Using same arguments in the previous paragraph, one has $v(M)$ is $\ell'$-subnormal in $G$ with $\ell\ge \ell'$. By Lemma~\ref{l1}, $\ell'\ge n$. Thus, $\ell\ge \ell'\ge n$. The theorem is complete.}
\end{proof}
%------------------------------------------------------------------------------------------------------

\section{Division rings with finite dimensional subdivision rings}

This section relates to the Kurosh problem for division rings which conjectures that algebraic division rings are locally finite, that is, every finitely generated subdivision ring is finite dimensional over its center (e.g., see \cite[Problem 7]{Pa_Ze_07}). In \cite{Pa_Ma_00}, Mahdavi-Hezavehi proposed a weak version of the Kurosh problem as follow:

\begin{con}\label{c4.2} {\rm \cite[Problem 9, P. 82]{Pa_Ma_00}} Every non-commutative division ring algebraic over its center contains a centrally finite non-commutative subdivision ring
\end{con}

There is a vast number of applications in case this conjecture holds. For example, the answer to \cite[P. 180]{her1} and \cite[Conjecture 1.1]{Pa_Go_Pa_15} is affirmative naturally for the class of algebraic division rings.

Although Conjecture~\ref{c4.2} is true for many cases such as when $D$ is right (or left) algebraic bounded degree over a maximal subfield $K$ \cite{aa2,Pa_BeDrSh_13} or $D$ contains a non-central torsion elements \cite[Proposition 2.2]{mai}, it is still open in general. The aim of this section is to give affirmative answer to Conjecture~\ref{c4.2} in other cases.
\begin{lem}\label{l4.1}
	Let $D$ be a division ring with center $F$. Assume that $K$ is a subfield of $D$ and $x$ is an element in $D$ which is left algebraic over $K$ and $xKx^{-1}\subseteq K$. If $L=K+Kx+ \dots +Kx^n+\dots$ is the left vector subspace of $D$ over $K$ generated by $\{x^n\mid n\in \mathbb{N} \}$, then $L$ is a finite dimensional subdivision ring of $D$.
\end{lem}
\begin{proof}{ If $n$ is the degree of the minimal polynomial of $x,$ then it is clear that $L = K + Kx + \dots + Kx^{n-1}$. Therefore, one can verify that, as a left vector space, $L$ is finite dimensional over $K.$ Now, we show that $L$ is indeed a centrally finite division subring of $D$. For every $k, l\in K$, one has $(kx^i)(\ell x^j) = kx^i\ell x^j = k(x^i\ell x^{-i})x^{i+j}\in Kx^{i+j}$. Hence, $L$ is closed under multiplication, which implies that $L$ is a subring of $D$.  Let $a\in L\backslash \{0\}$. Since $\dim_KL=n$, the subset $\{a^i\mid i\in \mathbb{N}\}$ is left linearly independent over $K$, which implies that $a$ is left algebraic over $K$. Let $a_0+a_1t+\dots +a_mt^m\in K[t]$ be the minimal polynomial of $a$. Then, $a_0+a_1a+\dots +a_ma^m=0$. Observe that $a_0\ne 0$, we have $$1=a a_0^{-1}(-a_{m}a^{m-1}-\dots -a_1)=a_0^{-1}(-a_{m}a^{m-1}-\dots -a_1)a.$$ As a corollary,  $a^{-1}$ belongs to $K+Ka+Ka^2+\dots \subseteq L$, so $L$ is a subdivision ring of $D$. Finally, since $L$ is finite dimensional over $K$ as a left vector space, by \cite[Lemma 6]{Akbari}, $L$ is centrally finite. The proof is complete.}	
\end{proof}

\begin{thm} Let $D$ be a division ring algebraic over its center $F$. One of the following assertions holds:
	\begin{enumerate}
		\item Every maximal subfield of $D$ is self-invariant.
		\item $D$ contains a centrally finite non-commutative subdivision ring.
	\end{enumerate}
\end{thm}
\begin{proof}
	{Assume that there exists a maximal subfield $K$ of $D$ and $x\in D\setminus K$ such that $xKx^{-1}\subseteq K$. Put $L=K+Kx+Kx^2+\cdots $, the vector subspace of the left vector space $D$ generated by $\{x^i\mid i\in \mathbb{N}\}$ over $K$. Observe that $x$ is algebraic over $F$, so $x$ is left algebraic over $K$. By Lemma~\ref{l4.1}, $L$ is a centrally finite subdivision ring of $D$. Since $K$ is a maximal subfield of $D$ and $x\in L\backslash K$, the division ring $L$ is non-commutative. The proof is complete.}
\end{proof}

\begin{thm}\label{thm2}
	Let $D$ be an algebraic division ring. If $D^*$ contains a non-abelian solvable subgroup, then $D$ contains a centrally finite non-commutative subdivision ring.
\end{thm}

\begin{proof}{ Let $F$ be the center of $D$. Suppose that $H$ is a non-abelian subgroup of $D^*$ such that
	$$H=H^{(0)} \rhd H^{(1)} \rhd H^{(2)} \rhd \cdots \rhd H^{(n-1)} \rhd H^{(n)}=1.$$
	Then, $n \ge 2$ as $H$ is non-abelian. Without loss of generality, we assume that $n=2$, that is, $H$ is metabelian. Let $S$ be the set of abelian subgroups of $H$ containing the derived subgroup $H'$. It is trivial that $S \ne \emptyset$. Assume that $$H_1 \le H_2 \le \cdots \le H_n \le \cdots$$ is a chain in $S$. If $A=\bigcup\limits_{i \in \mathbb{N}} H_i$, then $A\in S$, so $S$ satisfies Zorn's Lemma. Let $A$ be a maximal element in $S$.  Put $K=F(A)$, then $K$ is a subfield of $D$ because $A$ is abelian. Then, observe that $H$ is non-abelian, so $H\backslash K$ is non-empty. Let $x\in H\backslash A$. Put $L=K+Kx+Kx^2 + \cdots + Kx^n+ \dots$. We claim that $L$ is a centrally finite non-commutative subdivision ring of $D$.  Since $A$ contains $H'$, one has that $A$ is normal in $H$, which implies that $xAx^{-1}\subseteq A$. Hence, $xKx^{-1} \subseteq K$. By Lemma~\ref{l4.1}, $L$ is a centrally finite subdivision ring of $D$. Moreover, as $A$ is a maximal abelian subgroup of $H$ and $x\in H\backslash A$, one has $x\not\in C_H(A)$. which implies that $L$ is non-commutative.
	The proof is complete.} 		
\end{proof}

We give affirmative answer to \cite[Problem 9, P. 82]{Pa_Ma_00} in case $D$ contains a maximal subfield $K$ which is algebraic of bounded degree over $F.$
%_______________________________________________________________________________________________________________________
\begin{lem}\label{l1.1}{\rm \cite{Pa_AmRo_94}}
Let $D$ be a division ring with center $F$. If $a\in D$ is an inseparable algebraic element over $F$, then there exists $b\in D$ such that $ab-ba=1$.
\end{lem}
\begin{thm}
Let $D$ be a non-commutative division ring algebraic over its center $F$. If $D$ contains a maximal subfield $K$  algebraic of bounded degree over $F,$ then $D$ contains  a non-commutative subdivision ring which is finite-dimensional over $F.$
\end{thm}
\begin{proof}{There are two cases: (1) There exists $a\in K\backslash F$ which is inseparable over $F$; and (2) Every element in $K\backslash F$ is separable over $F$.
		
	{\bf Case 1.} Assume that there exists $a\in K\backslash F$ that is inseparable over $F$. By Lemma~\ref{l1.1}, there exists $b\in D$ such that $ab-ba=1$. Let $D_1=F(a,b)$ be the subdivision ring of $D$ generated by $a,b$ over $F$. It is clear that $D_1$ is non-commutative. Observe that $a,b$ are algebraic over $F$ and $ba=1-ab$, so that every element $x$ of $D_1$ can be written in the form of $$x=\sum_{i=1}^t\alpha_ia^{n_i}b^{m_i},$$ where $n_i,m_i\le 0, \alpha\in F$. It implies that $D_1$ is finite-dimensional over $F$.
	
	{\bf Case 2.} Assume every element in $K\backslash F$ is separable over $F$. Let $n={\rm degmax}_F(K)$ and $a\in K\backslash F$ such that $\deg_F(a)=n$.
	In this case, we will show that $D$ is centrally finite. Indeed, we first claim that $F(a)=K$.  If $K\backslash F(a)\ne \emptyset$, then fix an element $b\in K$ such that $b\notin F(a)$ and consider subfield $F(a,b)$. Then, $[F(a,b):F]=[F(a,b):F(a)][F(a):F]>n$. Since $a,b$ are separable, $F(a,b)$ is a simple extension of $F$ (\cite[Theorem 9.18*]{Bo_Hu_12}). Therefore, there exists $c\in F(a,b)$ such that $F(a,b)=F(c)$. Note that $[F(c):F]=[F(a,b):F]>n$, so $\deg_F(c)>n$, which is a contradiction! The claim is shown: $F(a)=K$. As a corollary, ${\rm dim}_FK=n<\infty$. Hence, $D$ is centrally finite.} 	
\end{proof}

Using the same argument in Case 2 of the previous theorem, we can show the following result.

\begin{thm}
Let $D$ be a division ring with center $F$. Assume that $D$ contains a maximal subfield $K$ which is algebraic of bounded degree over $F.$ If ${\rm Char}(F)$ is not a divisor of ${\rm degmax}_F(K),$ then $D$ is centrally finite. In particular, if ${\rm Char}(F)=0,$ then $D$ is centrally finite.
\end{thm}
\begin{proof}{Let $a$ be an element of $K$ such that $\deg_F(a)=n.$ Since $p$ is not a divisor of ${\rm Char}(F),$ the element $a$ is separable over $F$. Now using the same argument in Case 2 of the previous theorem, we have $D$ is centrally finite.}
\end{proof}
%--------------------------------------------------------------------------------------------------------------------------

\end{document}